\documentclass[preprint,12pt]{elsarticle}

\usepackage{amsfonts,amsthm,amsmath}
\usepackage{mathrsfs}

\usepackage{times}
\usepackage{amsfonts}
\usepackage{amsthm}
\usepackage{amsmath}
\usepackage{epic}
\usepackage{times}

\usepackage[margin=3cm]{geometry}
\usepackage{amsthm,amsmath}
\usepackage{fullpage}

\usepackage{tikz}
\usetikzlibrary{snakes}
\usepackage{soul}

\newcommand{\nats}{{\mathbb N}}

\newcommand{\card}{{\rm card}}
\def\A{\mathbb{A}}
\def\T{\mathbb{T}}

\newtheorem{thm}{Theorem}
\newtheorem{cor}{Corollary}
\newtheorem{defn}{Definition}

\newtheorem{lemma}{Lemma}

\newtheorem{proposition}[thm]{Proposition}
\newtheorem{question}{Question}

\newtheorem{example}{Example}

\theoremstyle{definition}

\begin{document}

\begin{frontmatter}

\title{On some variations of coloring problems of infinite words}

\author[label1]{Aldo de Luca}
  \ead{aldo.deluca@unina.it}

   \author[label2,label3]{Luca Q. Zamboni}
  \ead{lupastis@gmail.com}

\address[label1]{Dipartimento di Matematica e Applicazioni,
Universit\`a di Napoli Federico II, Italy}
\address[label3]{FUNDIM, University of Turku, Finland}
\address[label2]{Universit\'e de Lyon,
Universit\'e Lyon 1, CNRS UMR 5208,
Institut Camille Jordan,
43 boulevard du 11 novembre 1918,
F69622 Villeurbanne Cedex, France}

\begin{abstract}
Given a finite coloring (or finite partition) of the free semigroup $\A^+$ over a set $\A$, we consider   various types of monochromatic factorizations  of  right sided infinite words $x\in \A^\omega$.  Some stronger versions of the usual notion of monochromatic factorization are introduced.  A factorization is called sequentially monochromatic when  concatenations of consecutive blocks are monochromatic. A sequentially  monochromatic factorization is called ultra monochromatic if  any concatenation of arbitrary permuted blocks  of  the factorization has the same color of the single blocks. 
 We establish links, and in some cases equivalences, between the existence of these factorizations and fundamental results in Ramsey theory  including the infinite Ramsey theorem, Hindman's finite sums theorem,  partition regularity of IP sets and the Milliken-Taylor theorem. We prove that for each finite set $\A$ and each finite coloring $\varphi: \A^+\rightarrow C,$ for almost all words $x\in \A^\omega,$ there exists $y$ in the subshift generated by $x$ admitting a  $\varphi$-ultra monochromatic factorization, where ``almost all" refers to the Bernoulli measure on $\A^\omega.$ 
\end{abstract}

\journal{}

\end{frontmatter}

\section{Introduction and preliminaries}

Let   $\A$ be a  non-empty set, or {\em alphabet},  and  
 $\A^+$ denote the {\em free semigroup}  over $\A$, i.e.,
the set of all finite words $x=x_0x_1\cdots x_n$ with $x_i\in \A$, $0\leq i \leq n$. 
 Adding  to $\A^+$ an identity element $\varepsilon$, usually called {\em empty word}, one obtains the free monoid $\A^*$.
Let  $\A^\omega$  denote the set of all
right sided infinite words $x=x_0x_1\cdots $ with $x_i\in \A$, $i\geq 0$. For $x\in \A^\omega$ we let $\makebox{Fact}(x)=\{x_ix_{i+1}\cdots x_{i+j} \mid  i, j\geq 0\}$ denote the set of (non-empty) factors of $x$. 
A word $x \in \A^{\omega}$ is said to be (purely) {\em periodic} if  $x=u^\omega$, $u\in \A^+$, and  {\em ultimately periodic} if some suffix of $x$ is periodic. A word $x$ is called {\em aperiodic} if it is not ultimately periodic.

 Let $\varphi: \A^+ \rightarrow C$ be any mapping of $\A^+$ into a finite non-empty set $C$. We call the elements of $C$ {\em colors}
and $\varphi$ a {\em finite coloring}  of $\A^+$. 
We  consider three general notions of monochromatic factorization of $x$ relative to the coloring $\varphi$.

\begin{defn}\label{def:uno} Let  $\varphi: \A^+\rightarrow C$  be a finite coloring of $\A^+$  and  $x\in \A^\omega.$  A factorization $x=V_0V_1V_2\cdots$ with each $V_i\in \A^+$ is  called 
 
 \begin{itemize}
 \item  $\varphi$-{\it monochromatic} if $\exists c\in C$ such that $\varphi(V_i)=c$ \  for all $i\geq 0.$
 
 \item  $\varphi$-{\it sequentially monochromatic} if  $\exists c\in C$ such that $\varphi(V_iV_{i+1}\cdots  V_{i+j})=c$  \  for all $ i,j \geq 0.$ 
 
 \item $\varphi$-{\it ultra monochromatic}  if $\exists c\in C$ such that  for all $k\geq 1,$ and
all $0\leq n_1<n_2 <\cdots <n_k,$ and all permutations $\sigma $ of $\{1,2,\cdots ,k\}$ we have 
$\varphi(V_{n_{\sigma(1)}}V_{n_{\sigma(2)}}\cdots V_{n_{\sigma(k)}})=c.$ 
 \end{itemize}

\end{defn}
Clearly any $\varphi$-ultra monochromatic factorization is  $\varphi$-sequentially monochromatic and any   $\varphi$-sequentially monochromatic 
factorization is  $\varphi$-monochromatic. We begin with some examples.

Let  $\varphi: \A^+\rightarrow C$ be any finite coloring, and let  $x=u^\omega$, $u\in \A^+$,  be a periodic infinite word. Then the factorization $x=u\cdot u\cdot u\cdots$ is $\varphi$-monochromatic.  In general this factorization need not be $\varphi$-sequentially monochromatic. 

Let $\T=t_0t_1t_2\cdots \in \{0,1\}^\omega$ denote the {\it Thue-Morse infinite word}, where $t_n$ is defined as the sum modulo $2$ of the digits  in the binary expansion of
$n.$ \[\T=011010011001011010010\cdots \]The origins of $\T$ go back to the beginning of the last century with the works of  A. Thue \cite{Th1, Th2} in which he proves amongst other things that $\T$ is {\it overlap-free} i.e., contains no word of the form $uuu'$ where $u'$ is a non-empty prefix of $u.$ 

Consider  $\varphi: \{0,1\}^+\rightarrow \{0,1\}$  defined by $\varphi(u)=0$ if $u$ is a prefix of $\T$ and $\varphi(u)=1$ otherwise. It is easy to see that $\T$ may be factored uniquely as $\T=V_0V_1V_2\cdots$ where each $V_i\in\{0,01,011\}.$ Since each $V_i$ is a prefix of $\T,$ it follows that this factorization is $\varphi$-monochromatic. Since $V_1V_2=010$ is not a prefix of $\T,$ this factorization of $\T$ is not $\varphi$-sequentially monochromatic.  Next consider the coloring $\varphi ': \{0,1\}^+\rightarrow \{0,1,2\}$ defined by $\varphi '(u)=0$ if $u$ is a prefix of $\T$ ending with $0,$ $\varphi '(u)=1$ if $u$ is a prefix of $\T$ ending with $1,$ and $\varphi '(u)=2$ otherwise. We claim that $\T$ does not admit a $\varphi '$-monochromatic factorization. In fact, suppose to the contrary that $\T=V_0V_1V_2\cdots$ is a  $\varphi '$-monochromatic factorization. Since $V_0$ is a prefix of $\T$, it follows that there exists $a\in \{0,1\}$ such that each $V_i$ is a
  prefix of $\T$ terminating with $a.$ Pick $i\geq 1$ such that $|V_i|\leq |V_{i+1}|.$ Then  $aV_iV_i \in \makebox{Fact}(\T).$ Writing
$V_i=ua,$ (with $u$ empty or in $\{0,1\}^+),$ we see $aV_iV_i=auaua $ is an overlap, contradicting that $\T$ is overlap-free.

The following question\footnote{The original formulation of the question was stated in terms of finite colorings of $\mbox{Fact}(x)$ instead of $\A^+.$} was independently posed by T. Brown in \cite{BTC} and by the second author in \cite{LQZ}:

\begin{question}\label{conj} Let $x\in \A^\omega$ be non-periodic. Does there exist a finite coloring $\varphi: \A^+\rightarrow C$ relative to which $x$ does not admit a $\varphi$-monochromatic factorization?
\end{question}

Various partial results in support of an affirmative answer to this question were obtained in \cite{DZ, DPZ, ST}. In particular, it is shown that Question~\ref{conj} admits an affirmative answer for all non-uniformly recurrent words and various classes of uniformly recurrent words including Sturmian words. What is immediate to see  is that if $x\in \A^\omega$ is not periodic, then there exists a finite coloring $\varphi: \A^+\rightarrow \{0,1\}$ relative to which no factorization of $x$ is $\varphi$-sequentially monochromatic. In fact, it suffices to define $\varphi(u)=0$ if $u$ is a prefix of $x$ and $\varphi(u)=1$ otherwise. We claim that $x$ does not admit a $\varphi$-sequentially monochromatic factorization. In fact, suppose to the contrary that  $x=V_0V_1V_2\cdots$ is  a $\varphi$-sequentially monochromatic factorization. Then since $V_0$ is a prefix of $x,$ it follows that $\varphi(V_0)=0$ and hence $\varphi(V_iV_{i+1}\cdots V_{i+j})=0$ for each $i,j\geq 0.$ In particular taking $i=1$ we deduce that $V_1V_2\cdots V_j$ is a prefix of $x$ for each $j\geq 1.$ It follows that $x=V_1V_2V_3\cdots,$ and hence $x=V_0x,$ whence $x$ is periodic, a contradiction.

 In the next sections, we establish links, and in some cases equivalences, between the existence of the factorizations given in Definition \ref{def:uno} and fundamental results in Ramsey theory  including the infinite Ramsey theorem, Hindman's finite sums theorem,  partition regularity of IP sets, and the Milliken-Taylor theorem. One of the main results is  that for each finite set $\A$ and each finite coloring $\varphi: \A^+\rightarrow C,$ for almost all words $x\in \A^\omega,$ there exists $y$ in the subshift generated by $x$ admitting a  $\varphi$-ultra monochromatic factorization, where ``almost all" refers to the Bernoulli measure on $\A^\omega.$ 
 
 We conclude this section by introducing some notations and definitions which are relevant to subsequent sections.
 Given a set $S$ and a positive integer $k,$  let $\Sigma_k(S)$ denote the set of  $k$-element subsets of $S$  and $\makebox{Fin}(S)$  the set of all finite subsets of $S.$ 
We let $\nats=\{0,1,2,\ldots \}$ denote the set of natural numbers and $\nats^+=\nats \setminus \{0\}$ the set of positive integers. For $F,G\in \makebox{Fin}(\nats),$ we write $F<G$ if $\max(F)<\min(G).$

Let $x= x_1x_2 \cdots x_n$, $x_i\in \A$, $1\leq i\leq n$, be a word.  The quantity $n$ is called the {\em length} of $x$ and is denoted $|x|.$  The length of $\varepsilon$ is $0$.
For each word $x$ and $a\in \A$, we let $|x|_a$  denote the number of occurrences
of $a$ in $x$.
 The
\emph{reversal} of $x$ is the word $x^{\sim}= x_n\cdots x_1$.
A  factor $y$ of a finite or infinite word  $x$  is called {\em right special} (resp., {\em left special}) if there exist two different elements  $a$ and $b$ of $\A$ such that $ya$ and $yb$ (resp., $ay$ and $by$) are factors of $x$.

Let  $x\in \A^{\omega}$. The {\em factor complexity} of $x$ is the map $\lambda_x: \makebox{Fact}(x) \rightarrow \nats$ defined as follows: for any $n\geq 0$,
 $\lambda_x(n)$ counts the number of distinct factors of $x$ of length $n$.
An {\em occurrence} of $u \in \makebox{Fact} (x)$ in $x$ is any integer $n\geq 0$ such that  $x_nx_{n+1}\cdots x_{n+|u|-1}=u$. A factor $u$ of  $x \in \A^\omega$ is called {\em recurrent} if $u$ occurs infinitely many times in $x$  and {\em uniformly recurrent} if 
there exists an integer $k$ such that in any factor of $x$ of length $k$ there is at least one occurrence of $u$. An infinite word $x$ is called {\it recurrent} (resp., {\it uniformly recurrent}) if each of its factors is recurrent (resp., uniformly recurrent). As is well known \cite{F} for any infinite word $x$ there exists a uniformly recurrent word $y$ such that  $\makebox{Fact}(y) \subseteq \makebox{Fact}(x)$.

We endow $\A^\omega$ with the topology generated by the  metric
\[d(x, y)=\frac 1{2^n}\,\,\mbox{where} \,\, n=\min\{k \mid x_k\neq y_k\}\] 
whenever $x=(x_n)_{n\in \nats}$ and $y=(y_n)_{n\in \nats}$ are two
elements of $\A^\omega.$ 
The resulting topology is genera\-ted by the collection of {\it cylinders} $ [a_0,\ldots ,a_n]$, where
for each $n\geq 0$ and $a_i\in \A$, $0\leq i \leq n$,
\[
[a_0,\ldots ,a_n]=\{x\in \A^\omega \,|\, x_i=a_i\, \ \mbox{for}\, \ 0\leq i\leq n\}.\]
It can also be described as being the product topology on $\A^\omega$ with the discrete topology on  $\A$. In particular this topology is compact. The {\em Bernoulli measure} on $\A^{\omega}$, where $\A$ is finite,  is  defined as the unique measure $\mu$ on the $\sigma$-algebra of $A^\omega$ such that $\mu([a_0,\ldots ,a_n])= d^{-(n+1)}$  with $d=\card(\A)$.

Let $T:\A^\omega
\rightarrow \A^\omega$ denote the {\it shift}
transformation defined by $T: (x_n)_{n\in \nats}\mapsto
(x_{n+1})_{n\in \nats}.$ The {\em shift orbit} of $x$ is the set $\text{orb}(x)=\{T^k(x) \mid k\geq 0\}$,  i.e., the set of all suffixes of $x$. By a {\it subshift} on ${\A}$ we
mean a pair $(X,T)$ where $X$ is a closed and $T$-invariant subset
of $\A^\omega.$ A subshift $(X,T)$ is said to be {\it
minimal} whenever $X$ and the empty set are the only $T$-invariant
closed subsets of $X.$ With each $x \in \A^\omega$ is
associated the subshift $(X,T)$ where $X$ is the shift orbit
closure  of $x.$ This subshift, denoted by $\Omega(x)$, is usually called the {\em subshift generated by $x$}. As is well known (see, for instance, \cite[Theorem 10.8.9]{AS}) one has that
$$\Omega(x)=\{y \in \A^{\omega} \mid \makebox{Fact}(y)  \subseteq \makebox{Fact}(x) \}.$$
 If $x$ is uniformly recurrent, then $\Omega(x)$ is minimal, so that any two words $y$ and
$z$ in $\Omega(x) $ have exactly the same set of factors.

A word $x\in \{0,1\}^{\omega}$ is called {\it Sturmian} (cf., \cite[Chap. 2]{LO2})  if  it is aperiodic and {\em balanced}, i.e., for all factors $u$ and $v$ of $x$ such that $|u|=|v|$ one has
$$ | |u|_a-|v|_a| \leq 1, \ a\in \{0,1\}.$$
It follows that each Sturmian word contains exactly one of the two factors $00$ and $11$. Alternatively, a binary infinite word $x$ is Sturmian if  $x$ has a unique left (or equivalently right) special factor  of length $n$  for each  integer $n\geq 0$. This is equivalent to saying  that  for each $n\geq 0$ the number 
 of distinct factors of $x$ of length $n$ is exactly equal to $ n+1$. As a consequence one derives that  a Sturmian word $x$ is {\em closed under reversal}, i.e., if $u$ is a factor of $x$, then so is its reversal $u^{\sim}$ (see, for instance, \cite[Proposition 2.1.19]{LO2}).  The most famous Sturmian word is the Fibonacci word $f= 0100101001001010010\cdots$ which is the fixed point of the morphism $F$ defined by
 $F: 0 \mapsto 01, 1\mapsto 0$.
 
 For  $a\in \{0, 1\}$, we consider the injective endomorphism  $L_a$ of $\{0, 1\}^*$   defined by $  L_a : a\mapsto a,  b\mapsto ab .$
We recall \cite[Proposition 2.3.1]{LO2} that  the image $L_a(y)$ of any  Sturmian word $y$ is a Sturmian word. Moreover, for any word $y\in \{0, 1\}^{\omega}$ if $L_a(y)$ is a Sturmian word, then $y$ is also Sturmian \cite[Proposition 2.3.2]{LO2}.

An infinite word $x\in \A^{\omega}$ is  $r$-power free, $r>1$, if for each $u\in \makebox{Fact}(x)$ one has $u^r\not\in \makebox{Fact}(x)$.
For instance, the word $\T$ is $3$-power free and $f$ is $4$-power free (see, for instance, \cite{CD}).

\section{Main results}
 Given any finite coloring $\varphi$ of $\A^+$ and any infinite word $x\in \A^\omega,$ while it may happen as we have previously seen,  that $x$ does not admit a $\varphi$-monochromatic factorization, M. P. Sch\"utzenberger proved that for each finite coloring $\varphi$ of $\A^+$ and each infinite word $x\in \A^\omega$ there exists always a suffix of $x$ admitting a $\varphi$-monochromatic factorization (see \cite{schutz}).
The next theorem however provides a remarkable strengthening of Sch\"utzenberger's result.

\begin{thm}\label{ram} The following statements are equivalent:
\begin{enumerate}
\item For any  finite coloring $\varphi: \A^+\rightarrow C$ and any word $x\in \A^\omega, $  there exists a suffix $x'$ of $x$ which admits a  $\varphi$-sequentially monochromatic factorization. 

\item For each finite coloring $\varphi: \Sigma_2(\nats)\rightarrow C,$ there exist $c\in C$ and an infinite set
$\mathcal{N}\subseteq \nats$ such that $\Sigma_2(\mathcal{N})\subseteq \varphi^{-1}(c).$
\end{enumerate}
\end{thm}

\begin{proof} We note that item (2) is a special case of the Infinite Ramsey's Theorem (see \cite{GRS}). We begin by showing that $(2)\Longrightarrow (1).$
 Let $\varphi: \A^+\rightarrow C$ be any finite coloring, and $x=x_0x_1x_2\cdots \in \A^\omega.$   Then $\varphi$ induces a finite coloring $\varphi ': \Sigma_2(\nats)\rightarrow C$ given by $\varphi '(\{m<n\})=\varphi(x_mx_{m+1}\cdots x_{n-1}).$ By $(2)$ there exists $c\in C$ and an infinite
subset $\mathcal{N}=\{n_0<n_1<n_2<\cdots\}$ of $\nats$ such that for all $m,n \in \mathcal{N}$ with $m<n$ we have $\varphi '(\{m<n\})=c.$ It follows that the factorization of the suffix $x'=x_{n_0}x_{n_0+1}x_{n_0+2}\cdots$ given by $x'=V_0V_1V_2\cdots$ where $|V_i|=n_{i+1}-n_i$ is $\varphi$-sequentially monochromatic. 

To see that $(1)\Longrightarrow (2),$ let  $\varphi: \Sigma_2(\nats)\rightarrow C$ be any finite coloring of 
$ \Sigma_2(\nats).$ Let $x\in \{0,1\}^\omega$ be any aperiodic word. Then $\varphi$ induces a finite coloring $\varphi ': \A^+\rightarrow C\cup\{*\},$ where $*$ denotes a symbol not in $C,$ defined as follows: For each $u\in \A^+, $ if $u\notin \mbox{Fact}(x),$ then set $\varphi'(u)=*.$ Otherwise, 
let $m(u)$ be the least natural number $m$ such that $u=x_mx_{m+1}\cdots x_{m+|u|-1},$ that is $m(u)$ is the first occurrence of $u$ in $x.$ Then we put \[\varphi '(u)=\varphi (\{m(u), m(u)+|u|\}).\] By $(1)$ there exists $n\geq 0$ such that the suffix $x'=x_nx_{n+1}x_{n+2}\cdots$ of $x$ admits a $\varphi '$-sequentially monochromatic factorization $x'=V_0V_1V_2\cdots.$ Put $c=\varphi '(V_0).$ Since $V_0\in \mbox{Fact}(x)$ we have $c\in C.$ Also, as $x$ is aperiodic, there exists $s\geq 0$ such that $n=m(V_0V_1\cdots V_s).$  Indeed, set $x= Ux'$ with $|U|=n$. The statement is clear if $n=0$. Otherwise, if for each $s$,  $m(V_0V_1\cdots V_s)<n$, then 
by the pigeonhole principle there existsÊ Ê$0\leq k <n$ such that $k= m(V_0V_1\cdots V_s)$ for infinitely many values of $s$. This implies that $xÕ=T^k(x)=T^n(x)$, whence $xÕ$ is purely periodic and hence $x$ is ultimately periodic, a contradiction.

 Similarly, for each $r\geq 1$ there exists $s\geq r$ such that
$m(V_r\cdots V_s)=n+ \sum_{i=0}^{r-1}|V_i|.$ 
Given any  increasing sequence $0=n_0<n_1<n_2<\cdots ,$ put $W_k=V_{n_k}V_{n_k+1}\cdots V_{n_{k+1}-1}.$ Then clearly the factorization $x'=W_0W_1W_2\cdots$ is also $\varphi '$-sequentially monochromatic.
Thus we can assume that $x'$ admits a $\varphi '$-sequentially monochromatic factorization $x'=V_0V_1V_2\cdots,$ 
such that $m(V_0)=n$ and $m(V_r)=n+\sum_{i=0}^{r-1}|V_i|$ for each $r\geq 1.$ Setting \[\mathcal{N}=\{n<n+|V_0|<\cdots <n+\sum_{i=0}^r|V_i|<\cdots\},\] we have
$\Sigma_2(\mathcal{N})\subseteq \varphi^{-1}(c)$ as required. 
\end{proof}

\begin{example}\label{ex:uno} {\em Let  $x= 010110111011110 \cdots$. Consider the finite coloring $\varphi: \{0,1\}^+\rightarrow \{0,1\}$ defined by $\varphi(u)=0$ if $u$ is a prefix of $x,$ and $\varphi(u)=1$ otherwise. Following \cite[Lemma 3.4]{DPZ}, since $0$ is not uniformly recurrent in $x,$ it follows that $x$ does not admit a prefixal factorization, i.e., $x$ is not a concatenation of its prefixes. It follows that $x$ does not admit a $\varphi$-monochromatic factorization. In contrast, by Theorem~\ref{ram}, for every finite coloring $\varphi:\{0,1\}^+\rightarrow C,$ there exists a suffix of $x$ which admits a $\varphi$-sequentially monochromatic factorization. Finally, let $\varphi: \{0,1\}^+\rightarrow \{0,1\}$ be defined by  $\varphi(u)=0$ if $u$ is a factor of $x,$ and $\varphi(u)=1$ otherwise. We claim that no suffix of $x$ admits a $\varphi$-ultra monochromatic factorization. In fact, suppose to the contrary that some suffix $x'$ of $x$ admits a $\varphi$-ultra monochromatic factorization $x'=V_0V_1V_2\cdots.$ By concatenating several of the $V_i$  together (as in the proof of Theorem \ref{ram}), we can assume that each $V_i$  contains at least two occurrences of $0.$ Then,  $V_jV_i$ is not a factor of $x$ whenever $i<j.$  Thus $0=\varphi(V_0)\neq \varphi(V_jV_i)=1.$ }
\end{example}

The following  proposition illustrates  how in some very special cases, Theorem~\ref{ram} can be used to construct an ultra monochromatic factorization:

\begin{proposition} Let $C$ be a finite semigroup and $\varphi: \A^+ \rightarrow C$ a morphism. Let $x\in \A^{\omega}.$  There exists a suffix $x'$ of $x$
which admits a $\varphi$-ultra monochromatic factorization.
\end{proposition}
\begin{proof} By Theorem \ref{ram} there exists a suffix $x'$ of $x$ which admits a $\varphi$-sequentially monochromatic factorization
$ x'= V_0V_1\cdots V_n\cdots. $
Thus there exists $c\in C$ such that for all $i\geq 0$, one has  $\varphi(V_i\cdots V_{i+j})= c$, for all $i,j \geq 0$.
This implies $\varphi(V_0V_1)=c= \varphi(V_0)\varphi(V_1)=c^2$. Therefore, $c$ is an idempotent of the semigroup $C$. Thus $\varphi(u)=c$ for any $u\in \{V_i\,|\,i\geq 0\}^+.$ 
Whence the factorization $ x'= V_0V_1\cdots V_n\cdots $ is also  $\varphi$-ultra monochromatic.\end{proof}

In view of  the preceding results, it is natural to ask the following question:

\begin{question}\label{ultra}Let $\varphi:\A^+\rightarrow C$ be a finite coloring of $\A^+$ and  $x\in \A^\omega.$ Does there exist $y\in \Omega(x) $  admitting a $\varphi$-ultra monochromatic factorization? 
\end{question} 

It turns out that in general Question~\ref{ultra} does not admit an affirmative answer. 
We begin by exhibiting a $\varphi:\A^+\rightarrow C$ and $x\in \A^{\omega}$ such that no $y\in \Omega(x)$ admits a $\varphi$-ultra monochromatic factorization.

\begin{lemma}Let $r\in \nats^+$ and $x\in \{0,1\}^\omega$ be a $r$-power free Sturmian word. 
Then for each infinite sequence $\omega=V_0,V_1,V_2,\cdots $ with $V_i\in  \{0,1\}^+$ there exist $k\geq 1,$
$0\leq n_1<n_2<\cdots <n_k$, and a permutation $\sigma$ of $\{1,2,\ldots ,k\}$ such that
$V_{n_{\sigma(1)}}V_{n_{\sigma(2)}}\cdots V_{n_{\sigma(k)}} \notin{Fact}(x)$. 
\end{lemma} 

\begin{proof} For each $\omega=V_0,V_1,V_2,\cdots $ with $V_i\in  \{0,1\}^+,$ $i\geq 0$,  set
$N(\omega)=|V_0V_1\cdots V_r|.$ We proceed by induction on $N(\omega)$ to show that for each $\omega=V_0,V_1,V_2,\cdots $ with $V_i\in  \{0,1\}^+,$ and each $r$-power free Sturmian word $x$ 
there exist $k\geq 1,$ $0\leq n_1<n_2<\cdots <n_k$ and a permutation $\sigma$ of $\{1,2,\ldots ,k\}$ such that
$V_{n_{\sigma(1)}}V_{n_{\sigma(2)}}\cdots V_{n_{\sigma(k)}} \notin \mbox{Fact}(x).$

The base case of the induction is when $N(\omega)=r+1,$ i.e., $|V_0|=|V_1|=\cdots =|V_r|=1.$
Let $x$ be a $r$-power free Sturmian word. For $a\in \{0,1\},$ put $\bar{a}=1-a$ so that $\{a,\bar{a}\}=\{0,1\}.$ 
Fix $a\in \{0,1\}$ so that $\bar{a}\bar{a} \notin \mbox{Fact}(x).$    
First suppose that $V_i=V_j=\bar{a}$ for some $0\leq i<j.$   In this case $V_iV_j\notin \mbox{Fact}(x).$
Thus we can assume that at most one $V_i=\bar{a}.$ In this case, there exist  $0\leq n_1<n_2<\cdots <n_r\leq r$
such that $V_{n_i}=a$ for each $1\leq i\leq r.$ It follows that $V_{n_1}V_{n_2}\cdots V_{n_r}=a^r\notin \mbox{Fact}(x).$

For the inductive step, let $N>r+1,$ and suppose that for each  $\omega=V_0,V_1,V_2,\cdots $ with $V_i\in  \{0,1\}^+$ and $N(\omega)<N$ and for each $r$-power free  Sturmian word $x$ there exist $k\geq 1,$ $0\leq n_1<n_2<\cdots <n_k$, and a permutation $\sigma$ of $\{1,2,\ldots ,k\}$ such that
$V_{n_{\sigma(1)}}V_{n_{\sigma(2)}}\cdots V_{n_{\sigma(k)}} \notin \mbox{Fact}(x).$
Now let $\omega=V_0,V_1,V_2,\cdots $ with $V_i\in  \{0,1\}^+$, $i\geq 0$ and  $N(\omega)=N$ and let $x$ be a $r$-power free Sturmian word. Without loss of generality we may assume $11\notin \mbox{Fact}(x)$ and that $x$ begins with $0.$ 
Note that if $11\notin \mbox{Fact}(x)$ and $x$ begins with $1$, we can replace $x$ with $0x$ which is Sturmian and $r$-power free. We claim that for some $k\geq 1,$ and $0\leq n_1<n_2<\cdots <n_k$ and permutation $\sigma$ of $\{1,2,\ldots ,k\}$ we have $V_{n_{\sigma(1)}}V_{n_{\sigma(2)}}\cdots V_{n_{\sigma(k)}} \notin \mbox{Fact}(x).$ Suppose to the contrary that for every $k\geq 1,$ $0\leq n_1<n_2<\cdots <n_k$ and permutation $\sigma$ of $\{1,2,\ldots ,k\}$ we have $V_{n_{\sigma(1)}}V_{n_{\sigma(2)}}\cdots V_{n_{\sigma(k)}} \in \mbox{Fact}(x).$ Since $x$ is $r$-power free, we have $\limsup_{n\rightarrow \infty}|V_n|=+\infty.$ 

Suppose first that for some $a\in \{0,1\}$ there exist $0\leq i<j$ such that $V_i$ begins with $a$ and $V_j$ begins with $\bar{a}.$ Pick $j<m<n$ such that $|V_{n}|>r|V_{m}|.$ Since $V_{m}V_{n}V_i, $ $V_{m}V_{n}V_j, $ $V_{n}V_{m}V_i,$ and $V_{n}V_{m}V_j$ are each factors of $x,$ it follows that each of $V_{m}V_{n}$ and $V_{n}V_{m}$ is a right special factor of $x.$ But since $|V_{m}V_{n}|=|V_{n}V_{m}|$ and $x$ has exactly one right special factor of each length, it follows that $V_{m}V_{n}=V_{n}V_{m},$
 from which one easily derives that $V_{m}^r$ is a prefix of $V_n$ and hence in particular $V_{m}^r\in \mbox{Fact}(x),$ a contradiction.  Thus we may suppose that all $V_i$ begin with the same letter $a\in \{0,1\}.$ A similar argument shows that all $V_i$ terminate with the same letter $b\in \{0,1\}.$ Moreover, as
$11\notin \mbox{Fact}(x),$ either $a$ or $b$ must equal $0.$ Since $\mbox{Fact}(x)$ is closed under reversal, short of replacing each $V_i$ in $\omega$ by its reversal, we may suppose that $a=0,$ i.e., each $V_i$ begins with $0.$  
Thus $V_i0\in \mbox{Fact}(x)$ for each $i\geq 0.$

 Now consider the morphism $L_0: 0\mapsto 0,$ and $1\mapsto 01.$ For each $i\geq 0,$ define $V_i'\in \{0,1\}^+$ by $L_0(V_i')=V_i$ and put $\omega '=V_0',V_1',V_2',\ldots .$ Finally, as $x$ begins with $0$, define $x'\in \{0,1\}^\omega$ by $L_0(x')=x.$ Then, as is well known, $x'$ is a Sturmian word. Moreover, since  $x$ is $r$-power free, so is $x'$ and at least one $V_i$ with $0\leq i\leq r-1$ must contain an occurrence of $1.$ Thus $N(\omega')<N(\omega).$ 
For each $k\geq 1,$ $0\leq n_1<n_2<\cdots <n_k$ and  permutation $\sigma$ of $\{1,2,\ldots ,k\},$ 
we have $V_{n_{\sigma(1)}}V_{n_{\sigma(2)}}\cdots V_{n_{\sigma(k)}}0 \in \mbox{Fact}(x).$ 
Thus $V'_{n_{\sigma(1)}}V'_{n_{\sigma(2)}}\cdots V'_{n_{\sigma(k)}} \in \mbox{Fact}(x'),$ and this is a contradiction to  our inductive hypothesis.
\end{proof}
We mention that very recently Anna Frid \cite{AF} has extended the validity of previous lemma to the case of any infinite word of linear factor complexity.

\begin{proposition}\label{CE}Let $r\in \nats^+$ and $x\in \{0,1\}^\omega$ be a $r$-power free Sturmian word. Define $\varphi:\{0,1\}^+\rightarrow \{0,1\}$ by $\varphi(u)=0$ if $u$ is a factor of $x$ and $\varphi(u)=1$ otherwise. 
Then no $y\in \Omega(x)$ admits a $\varphi$-ultra monochromatic factorization.

\end{proposition}

\begin{proof}Let $y\in \Omega(x)$ and consider any factorization $y=V_0V_1V_2\cdots .$ Then since $V_0$ is a factor of $x,$ we have that $\varphi(V_0)=0.$ On the other hand, since $y$ is also Sturmian and $r$-power free, by the previous lemma there exist $k\geq 1,$ $0\leq n_1<n_2<\cdots <n_k$ and a permutation $\sigma$ of $\{1,2,\ldots ,k\}$ 
such that $V_{n_{\sigma(1)}}V_{n_{\sigma(2)}}\cdots V_{n_{\sigma(k)}}\notin \mbox{Fact}(y).$ Hence 
$\varphi( V_{n_{\sigma(1)}}V_{n_{\sigma(2)}}\cdots V_{n_{\sigma(k)}})=1$ from which it follows that no $y\in \Omega(x)$ admits a $\varphi$-ultra monochromatic factorization.
\end{proof}

We next show (cf. Corollary \ref{cor: bern})  that if $\A$ is finite, then Question~\ref{ultra} admits an affirmative answer for almost all $x\in \A^\omega,$ where ``almost all" refers to the Bernoulli measure on $\A^\omega.$  We begin by showing that Question~\ref{ultra} admits a positive answer in case $x$ is periodic. Even this simplest case however turns out to be somewhat nontrivial, and in fact is equivalent to the so-called Finite Sums Theorem proved by N. Hindman in \cite{H}.

\begin{thm} \label{Hin}The following statements are equivalent:
\begin{enumerate} 
\item For every finite coloring $\varphi: \A^+\rightarrow C,$ each periodic word $x\in \A^\omega$ admits a $\varphi$-ultra monochromatic factorization.

\item For each finite coloring $\varphi:\nats^+\rightarrow C$  of the positive integers, there exist
$c\in C$ and an infinite sequence $(n_k)_{k=0}^\infty$ such that
$\makebox{FS}((n_k)_{k=0}^\infty)=\{\sum_{i\in F}n_i\,| F\in \makebox{Fin}(\nats)\}\subseteq \varphi^{-1}(c)$.
\end{enumerate}
\end{thm}

\begin{proof} We  note that item $(2)$ is the Finite Sums Theorem  by  Hindman. 
We begin by showing that $(1)\Longrightarrow (2).$ Let $\varphi:\nats^+\rightarrow C$ be a finite coloring of the positive integers, and let $x$ be the periodic word $x=a^\omega$, with $a\in \A$.  Then $\varphi$ induces a finite coloring $\varphi ': \{a\}^+\rightarrow C$ given by $\varphi ' (a^n)=\varphi (n).$ By $(1)$ there exists a $\varphi'$-ultra monochromatic factorization $x=V_0V_1V_2\cdots.$ Put $c=\varphi'(V_0).$ For $k\geq 0,$ set $n_k=|V_k|$ so that each $V_k=a^{n_k}.$  Then for each finite subset $F$ of $\nats,$ we have
$$\varphi (\sum_{i\in F}n_i)=\varphi(\sum_{i\in F}|V_i|)=\varphi(|\prod _{i\in F}V_i|)=\varphi'(\prod _{i\in F}V_i)=c$$ since $\prod _{i\in F}V_i=a^{\sum_{i\in F}n_i}$ and hence is a factor of $x.$  Whence $\makebox{FS}((n_k)_{k=0}^\infty)=\{\sum_{i\in F}n_i\,| F\in \makebox{Fin}(\nats)\}\subseteq \varphi^{-1}(c).$

To see that $(2)\Longrightarrow (1),$ let $\varphi: \A^+\rightarrow C,$ $u\in \A^+,$  and $x=u^\omega.$  Define  $\varphi':\nats^+\rightarrow C$ by $\varphi '(n)=\varphi (u^n).$ 
By $(2)$ there exist $c\in C$ and an infinite sequence $(n_k)_{k=0}^\infty$ such that
$\makebox{FS}((n_k)_{k=0}^\infty)=\{\sum_{i\in F}n_i\,| F\in \makebox{Fin}(\nats)\}\subseteq \varphi'^{-1}(c).$ For each $k\geq 0$ set $V_k=u^{n_k}.$ Then clearly the factorization $x=V_0V_1V_2\cdots$ is $\varphi$-ultra monochromatic. 
\end{proof}

\noindent As an immediate consequence we obtain:

\begin{cor}\label{cor: bern}Let $\A$ be a finite set, and let $\mu$ be the Bernoulli measure on $\A^\omega.$ Let $\varphi: \A^+\rightarrow C$ be any finite coloring. Then for $\mu$-almost all $x\in \A^\omega$ 
there exists $y\in \Omega(x)$ which admits a $\varphi$-ultra monochromatic factorization. 
\end{cor}

\begin{proof}As  is well known  almost all words $x\in \A^\omega$ with respect to the measure $\mu$,  are of full complexity, meaning
$\mbox{Fact}(x)=\A^+$ (see, for instance, \cite[Theorem 10.1.6]{AS}). (As an example, {\em normal words} \cite[Chap. 8]{NI} are of full complexity). Thus for almost all words $x\in \A^\omega,$ relatively to measure $\mu$,  there exists $a\in \A$ such that $a^\omega\in \Omega(x).$  The result now follows from Theorem~\ref{Hin}.
\end{proof}

The above results suggest the following questions:

\begin{question}Let $x\in \A^\omega$ be a uniformly recurrent word. Suppose that for each $\varphi: \A^+\rightarrow C,$ there exists $y\in \Omega(x)$ admitting a $\varphi$-ultra monochromatic factorization. Then does it follow that $x$ is periodic?
\end{question} 

Let $\varphi: \A^+\rightarrow C$ and $x\in \A^\omega.$ A factorization $x=V_0V_1V_2\cdots$ with $V_i\in\A^+$, $i\geq 0$, is  called $\varphi$-{\it conditionally monochromatic} if $\exists c\in C$ such that $\forall k\geq 1,$ for all $n_1<n_2<\cdots <n_k$ and for all permutations $\sigma$ of $\{1,2,\ldots ,k\}$ we have either $V_{n_{\sigma(1)}}V_{n_{\sigma(2)}}\cdots V_{n_{\sigma(k)}} \notin \mbox{Fact}(x)$ or $\varphi(V_{n_{\sigma(1)}}V_{n_{\sigma(2)}}\cdots V_{n_{\sigma(k)}})=c.$ 

Thus a $\varphi$-ultra monochromatic factorization of a word $x\in \A^\omega$ is a $\varphi$-conditionally monochromatic factorization, but not  {\em vice versa}. For instance, consider  $x= 010110111011110  \cdots$. We saw in  Example \ref{ex:uno}
that relative to the coloring $\varphi: \{0,1\}^+\rightarrow \{0,1\}$  defined by  $\varphi(u)=0$ if $u$ is a factor of $x,$ and $\varphi(u)=1$ otherwise, no suffix of $x$ admits a $\varphi$-ultra monochromatic factorization. On the other hand given any coloring $\varphi$ of
 $\{0,1\}^+$, by Theorem \ref{ram} there exists a suffix $x'$  admitting a $\varphi$-sequentially monochromatic factorization $ x=V_0V_1V_2 \cdots$.  By concatenating several of the $V_i$  together (as in the proof of Theorem \ref{ram}), we can assume that each $V_i$  contains at least two occurrences of $0$. The resulting $\varphi$-sequentially monochromatic factorization is then also $\varphi$-conditionally monochromatic since the only concatenation of blocks which yields a factor of $x$ are consecutive concatenations.

\begin{question} Let $\varphi:\A^+\rightarrow C$ be a finite coloring of $\A^+$ and  $x\in \A^\omega.$ Does there exist $y\in \Omega(x) $  admitting a $\varphi$-conditionally monochromatic factorization? 
\end{question}

We have not a single example of an aperiodic uniformly recurrent word in which we can give an answer (positive or negative) to the above question.

\section{Shift invariant monochromatic factorizations}

Let $\varphi: \A^+\rightarrow C$ be a finite coloring of $\A^+,$    
$x\in \A^\omega,$ and $k$ be a positive integer. A $\varphi$-monochromatic (resp., $\varphi$-sequentially monochromatic, $\varphi$-ultra monochromatic) factorization  $x=V_0V_1V_2\cdots $ is said to be $k$-{\it shift invariant} if for each $1\leq j\leq k$ the
induced factorization $T^j(x)=W_0W_1W_2\cdots $ with $|W_i|=|V_i|$, $i\geq 0$, is $\varphi$-monochromatic (resp., $\varphi$-sequentially monochromatic, $\varphi$-ultra monochromatic). A $\varphi$-monochromatic (resp., $\varphi$-sequentially monochromatic, $\varphi$-ultra monochromatic) factorization  $x=V_0V_1V_2\cdots $ is called {\it shift invariant} if for each positive integer $j,$ the induced factorization $T^j(x)=W_0W_1W_2\cdots $ with $|W_i|=|V_i|$ is $\varphi$-monochromatic (resp., $\varphi$-sequentially monochromatic, $\varphi$-ultra monochromatic).\\

\noindent We begin with the following simple variation of the infinite Ramsey  theorem, whose proof is omitted as it is a simple iterated application of the usual version of Ramsey's theorem.

\begin{proposition}Let $\varphi: \Sigma_2(\nats)\rightarrow C$ be a finite coloring and $k$ a nonnegative integer. There exists an infinite set $\mathcal{N}\subseteq \nats$ and a sequence $(c_i)_{i=0}^k$ such that for each $0\leq i\leq k$ we have $c_i\in C$ and
\[ \Sigma_2(\mathcal{N}+i)\subseteq \varphi^{-1}(c_i).\]
\end{proposition}

\noindent As an immediate consequence we deduce that

\begin{cor}Let $\varphi :\A^+\rightarrow C$, $x\in \A^\omega,$ and  $k\geq 1.$ Then there exists a suffix $x'$ of $x$ which admits a $k$-shift invariant $\varphi$-sequentially monochromatic factorization. 
\end{cor} 

\begin{proof} As in the proof of Theorem~\ref{ram}, we  apply the above variation of Ramsey's theorem to the coloring $\varphi':\Sigma_2(\nats)\rightarrow C$ given by $\varphi '(\{m<n\})=\varphi(x_{m}x_{m+1}\cdots x_{n-1}).$
\end{proof} 

\begin{proposition} A word $x\in \A^\omega$ is ultimately periodic  if and only if for every finite coloring
$\varphi: \A^+\rightarrow C$ there exists a suffix of $x$ which admits a shift invariant $\varphi$-monochromatic factorization. 
\end{proposition} 

\begin{proof} Clearly, if $x$ is ultimately periodic, and hence of the form $x=uv^\omega$ for some $u, v\in \A^*$ with $v\neq \varepsilon$, then for any $\varphi: \A^+\rightarrow C,$ the factorization $v\cdot v\cdot v \cdot \cdots$ of the suffix $v^\omega$ is shift invariant $\varphi$-monochromatic. Conversely, suppose $x$ is aperiodic. 
Choose a recurrent word $y\in \Omega(x).$ Thus each prefix of $y$ occurs infinitely often in $x.$ Let $\varphi :\A^+\rightarrow \{0,1\}$ be given by $\varphi (u)=0$ if $u$ is a prefix of $y$ and $\varphi(u)=1$ otherwise. Let $x'$ be any suffix of $x.$ We claim that $x'$ does not admit a shift invariant $\varphi$-monochromatic factorization. In fact, suppose to the contrary that $x'$ admits a shift invariant $\varphi$-monochromatic factorization $x'=V_0V_1V_2\cdots.$ Since $y$ is recurrent  and each prefix of $y$ occurs infinitely often in $x,$ there exist $0\leq i<j$ such that if we consider  the shifted factorizations $T^i(x')=W_0W_1W_2\cdots$ and $T^j(x')=W'_0W'_1W'_2\cdots,$ where $|W_i|=|W'_i|=|V_i|$ for each $i\geq 0$, both $W_0$ and $W'_0$ are prefixes of $y.$ It follows that  $W_i$ and  $W'_i$ are prefixes of $y$ for each $i\geq 0.$ But since they are of equal length, we have $
 W_i=W'_i$ for each $i\geq 0.$ Thus
$T^i(x')=T^j(x')$ which implies that $x$ is ultimately periodic, a contradiction.
\end{proof}

We recall that a subset $A$ of $\nats^+$ is called an {\it IP set} if $A$ contains $\makebox{FS}((n_i)_{i=1}^\infty) $ for some infinite sequence
$(n_i)_{i=1}^\infty.$ In terms of IP sets, Hindman's theorem states that any finite coloring of  $\nats^+$ contains a monochromatic IP set.
By using the so-called {\em Finite Unions Theorem}, which is equivalent to Hindman's Finite Sums Theorem (cf. \cite{Bau,Mi}), one can show that
IP sets in $\nats^+$ are {\em partition regular}, i.e.,  if $A$ is an IP set and  $A=\bigcup_{i=1}^kA_i,$ then there exists  $1\leq i \leq k$ such that $A_i$ is an IP set. We recall also the following well-known theorem of Milliken-Taylor \cite{Mi, Ta}:
\begin{thm}\label{thm:MT} Let $k$ be a positive integer and $\varphi: \Sigma_k(\nats^+)\rightarrow C$ a finite coloring. Then there exist $c\in C$ and an infinite sequence $(n_i)_{i=1}^\infty$ such that
\[\{\sum_{i\in F_1}n_i,\sum_{i\in F_2}n_i, \ldots ,\sum_{i\in F_k}n_i\}\in \varphi^{-1}(c)\] for each $F_1<F_2<\cdots <F_k$ with $F_i\in \makebox{Fin}(\nats^+)$, $1\leq i \leq k$.
\end{thm}

The next theorem shows that for each finite coloring $\varphi: \A^+\rightarrow C,$  and each periodic word $x\in \A^\omega$  there exists a shift invariant $\varphi$-ultra monochromatic factorization of $x.$ We present two proofs, one uses the fact that IP sets are partition regular and the other uses the Milliken-Taylor Theorem.

\begin{thm}\label{SIF} 
For each finite coloring $\varphi: \A^+\rightarrow C,$  each periodic word $x\in \A^\omega$  admits a shift invariant $\varphi$-ultra monochromatic factorization.
\end{thm}

\begin{proof} ({\em First proof}) Let $\varphi: \A^+\rightarrow C$  be given. Let $u=u_1u_2\cdots u_k\in \A^+$, $u_i \in \A$, $i=1, \ldots, k$ and $x=u^\omega.$
Consider the coloring $\varphi_1: \nats^+\rightarrow C$ defined by $\varphi_1(n)=\varphi(u^n).$ Then by Hindman's theorem there exists an infinite sequence $(n^{(1)}_i)_{i=1}^\infty$ and $c_1\in C$ such that $\makebox{FS}((n^{(1)}_i)_{i=1}^\infty)\subseteq \varphi_1^{-1}(c_1).$ This implies that the factorization \[x=u^{n^{(1)}_1}\cdot u^{n^{(1)}_2}\cdot u^{n^{(1)}_3}\cdots\] is $\varphi$-ultra monochromatic. Next consider the coloring $\varphi_2:\makebox{FS}((n^{(1)}_i)_{i=1}^\infty) \rightarrow C$ defined by $\varphi_2(n)=\varphi((u_2\cdots u_ku_1)^n).$ By partition regularity of IP sets it follows that there exists an infinite sequence $(n^{(2)}_i)_{i=1}^\infty$ and $c_2\in C$ such that $\makebox{FS}((n^{(2)}_i)_{i=1}^\infty)\subseteq \varphi_2^{-1}(c_2).$ It follows that the factorizations \[x=u^{n^{(2)}_1}\cdot u^{n^{(2)}_2}\cdot u^{n^{(2)}_3}\cdots\] and \[T(x)=(u_2\cdots u_ku_1)^{n^{(2)}_1} (u_2\cdots u_ku_1)^{n^{(2)}_2} (u_2\cdots u_ku_1)^{n^{(2)}_3}\cdots\] are both 
 $\varphi$-ultra monochromatic.  Continuing in this way up to stage $k,$ we can find an infinite sequence
$(n^{(k)}_i)_{i=1}^\infty$  such that for each $0\leq i\leq k-1$ the factorization
\[T^i(x)=(u_{i+1}\cdots u_ku_1\cdots u_{i})^{n^{(k)}_1}(u_{i+1}\cdots u_ku_1\cdots u_{i})^{n^{(k)}_2}(u_{i+1}\cdots u_ku_1\cdots u_{i})^{n^{(k)}_3}\cdots\]
is $\varphi$-ultra monochromatic.  Since $T^k(x)=x$ the result now follows.

\vspace{2 mm}

\noindent({\em Second proof}) As before let $\varphi: \A^+\rightarrow C$  be given, $u=u_1u_2\cdots u_k\in \A^+$, $u_i \in \A$, $i=1, \ldots, k$, and $x=u^\omega.$ Then $\varphi$ induces a finite coloring
\[\Psi: \Sigma_k(\nats^+)\rightarrow C^k\] defined by
\[\Psi(\{n_1<n_2<\cdots<n_k\})=(\varphi((u_1u_2\cdots u_k)^{n_1}), \varphi((u_2u_3\cdots u_ku_1)^{n_2}), \ldots ,\varphi((u_ku_1\cdots u_{k-1})^{n_k}).\] 
By Theorem~\ref{thm:MT} there exist $c=(c_1,c_2,\ldots ,c_k)\in C^k$ and $(n_i)_{i=1}^\infty$ such that 
\begin{equation}\Psi(\{\sum_{i\in F_1}n_i,\sum_{i\in F_2}n_i, \ldots ,\sum_{i\in F_k}n_i\})=c\tag{$*$}\end{equation} for each $F_1<F_2<\cdots <F_k$ with $F_i\in \makebox{Fin}(\nats^+)$, $1\leq i \leq k$.

Fix $1\leq j\leq k$ and $F\in \makebox{Fin}(\{k, k+1, k+2,\ldots\}).$ We claim that \[\varphi((u_j\cdots u_ku_1\cdots u_{j-1})^{\sum_{i\in F}n_i})=c_j.\]
This is a consequence of $(*)$ by taking $F_i=\{i\}$ for $1\leq i<j,$ $F_j=F,$ and $F_{j+i}=\{M+i\}$ for $1\leq i\leq k-j$ where $M=\max(F).$  
It follows that the factorization
$x=u^{n_k}u^{n_{k+1}}u^{n_{k+2}}\cdots$ is shift invariant $\varphi$-ultra monochromatic.
\end{proof}

\vspace{2 mm}

 \noindent
{\bf Acknowledgments:} We  are  indebted to Neil Hindman for his suggestions and we thank the two anonymous reviewers for their  useful comments.

\bibliographystyle{plain}
\bibliography{RSZshort}

\end{document}